\documentclass[12pt,leqno]{article}
\usepackage[OT2,T1]{fontenc}
\usepackage[utf8]{inputenc}
\usepackage[top = 3.5cm, bottom  = 3cm, left = 3.7cm, right = 3.2cm]{geometry}
\usepackage{amsmath, amscd, amssymb, amsthm, amsfonts, euscript, mathtools}
\usepackage{tikz}
\usetikzlibrary{cd,babel}
	\tikzcdset{
  		cells={font=\everymath\expandafter{\the\everymath\displaystyle}},
	}
\usepackage{adjustbox}
\usepackage{cite}
\usepackage{url}
\usepackage{tocloft}

\usepackage{hyperref}
    \hypersetup{
        colorlinks=true,
        linktoc= all,     
        citecolor= black,
        linkcolor = black
    }

\usepackage{setspace}

\usepackage{graphicx}

\usepackage{authblk}

\DeclareSymbolFont{cyrletters}{OT2}{wncyr}{m}{n}
\DeclareMathSymbol{\Sha}{\mathalpha}{cyrletters}{"58}
\DeclareMathSymbol{\Bcyr}{\mathalpha}{cyrletters}{"42}

\newcommand{\N}{\mathbb{N}}
\newcommand{\Z}{\mathbf{Z}}

\renewcommand{\C}{\mathbf{C}}

\renewcommand{\H}{\mathrm{H}}

\DeclareMathOperator{\Gal}{Gal}

\DeclareMathOperator{\id}{id}
\DeclareMathOperator{\Hom}{Hom}

\DeclareMathOperator{\Spec}{Spec}

\DeclareMathOperator{\Inf}{Inf}
\DeclareMathOperator{\Res}{Res}

\DeclareMathOperator{\Bad}{\mathrm{Bad}}

\usepackage{lipsum}
\makeatletter
\def\@makechapterhead#1{%
  \vspace*{10\p@}%
  {\parindent \z@ \raggedright \normalfont
    \ifnum \c@secnumdepth >\m@ne
      \if@mainmatter
        \Huge\bfseries \thechapter.\space%
      \fi
    \fi
    \interlinepenalty\@M
    \Huge \bfseries #1\par\nobreak
    \vskip 40\p@
  }}
\makeatother

\expandafter\let\expandafter\oldproof\csname\string\proof\endcsname
\let\oldendproof\endproof
\renewenvironment{proof}[1][\proofname]{%
  \oldproof[\bfseries   #1]%
}{\oldendproof}

\theoremstyle{definition}
\newtheorem{thm}{Theorem}[section]
\newtheorem{thmx}{Theorem}

\newtheorem{defi}[thm]{Definition}
\newtheorem*{defi*}{Definition}

\newtheorem{prop}[thm]{Proposition}
\newtheorem*{prop*}{Proposition}
\newtheorem{lemma}[thm]{Lemma}

\newtheorem{rmk}[thm]{Remark}
\title{Reocurrence and weak approximation over geometric global fields}
\author{ Felipe Gambardella\\
\vspace*{-1ex}\small \textit{Centre de Math\'ematiques Laurent Schwartz - \'Ecole Polytechnique} \\
\texttt{felipe.gambardella@polytechnique.edu}}
\date{}

\begin{document}

\maketitle
\begin{abstract}
    In this article, we prove a Reocurrence Theorem over function fields of curves over $\C(\! (t)\! )$ and over finite extensions of the Laurent series field $\C(\! (x,y)\! )$. This provides a partial replacement to Chebotarev's Theorem over such fields. A concrete application to the study of weak approximation for homogeneous spaces under $\mathrm{SL}_n$ and with finite stabilizers is given at the end of the article.
\end{abstract}
\section{Introduction}

Let $K$ be a number field and let $\Omega_K$ be the set of places of $K$. In order to study rational points on $K$-varieties, it is often useful to introduce the local-global principle and weak approximation. A family of $K$-varieties $\mathcal{F}$ is said to satisfy the local-global principle if any variety $X \in \mathcal{F}$ that has points in all the completions of $K$ has rational points. It is said to satisfy weak approximation if, for any $X \in \mathcal{F}$, the set of rational points $X(K)$ is dense in the product $\prod_{v \in \Omega_K} X(K_v)$. Some families of varieties do satisfy these properties -for instance quadrics according to the Hasse-Minkowski theorem- while others do not. 

In recent years, there has been a growing interest in similar questions over some two-dimensional fields naturally arising in geometry. Consider a field $K$ of one of the following two kinds:
\begin{itemize}
    \item[(a)] the function field of a smooth projective curve $C$ over $\C(\! (t)\! )$;
    \item[(b)] a finite extension of the Laurent series field $\C(\! (x,y)\! )$ in two variables over the complex numbers. 
\end{itemize}
Rational points on varieties over such a field $K$ have been studied both by cohomological techniques based on duality theorems for Galois cohomology (e.g. \cite{CTH2015DualCt}, \cite{diego2019dual2dim}, \cite{DiegoLuco2021LocalGlobalHomoGGF} and \cite{Haowen2023WAHomoGGF}) and by patching techniques (e.g. \cite{HHK2015lgTorsorsArithmeticCurves} and \cite{CTHHKPS2019LocalGlobalZeroCyclesFunctionFields}).

Many results concerning the local-global principle and weak approximation over number fields crucially rely on Chebotarev's Theorem - more precisely on the non-existence of non-trivial everywhere locally trivial cyclic extensions. This result turns out to fail over our field $K$ (see \cite[Remarque 5.4]{CTPS2016LoisReciproSup} and \cite[Theorem 5]{Jaw2001StrongHasse}). The main purpose of this article consists in proving a reocurrence theorem that is, in some situations, a good replacement to Chebotarev's Theorem for the field $K$.

In order to state it, let us introduce the set of places of the field $K$ we will take into account. In case (a), when $K$ is the function field of a smooth projective curve $C$ over $\C(\! (t)\! )$, we consider the set $C^{(1)}$ of places induced by a closed point of the curve $C$. In case (b), when $K$ is a finite extension of the Laurent series field $\C(\! (x,y)\! )$, we introduce the integral closure $A$ of $\C[\! [x,y]\! ]$ in $K$, we let $\mathfrak{m}$ be the maximal ideal of $A$, we set $C:=\Spec A \setminus \{\mathfrak{m}\}$ and we consider the set $C^{(1)}$ of places induced by a closed point of the Dedekind scheme $C$. 

    \begin{thmx}[Theorem \ref{thm reocurrence}] \label{thm reocurrence intro}
         Let $K$ be a field of type (a) or (b) as above. For each finite Galois extension $L/K$ with group $G$ and  each place $v\in C^{(1)}$ that is unramified in $L/K$, there exists infinitely many places $w \in C^{(1)}$ such that the decomposition groups of $v$ and $w$ are conjugates as subgroups of $G$.
    \end{thmx}

    Note that, contrary to the classical Chebotarev's density Theorem over number fields, this result does not involve any Frobenius element. A finer version of Theorem \ref{thm reocurrence intro}, that provides a suitable replacement of Frobenius elements, is given in Section \ref{sec finer reocurrence} (see Theorem \ref{thm finer reocurrence}). \\

    At the end of the article, we provide a concrete application of our Reocurrence Theorem to the study of rational points on homogeneous spaces with finite stabilizers over the field $K$. Given a finite subset $S$ of $C^{(1)}$ and a finite $K$-group $F$, we say that $F$ has approximation away from $S$ if, for every embedding of $F$ into a semi-simple simply connected linear algebraic $K$-group $G$, the set $G/F(K)$ is dense in $\prod_{v \in C^{(1)}\setminus S} G/F(K_v)$. We denote by $\Bad_F$ the set of places $v \in C^{(1)}$ such that the minimal extension splitting $F$ ramifies at $v$. We then prove the following result:

    \begin{thmx}[Theorem \ref{thm weak approx split extension}] \label{thm weak approx intro}
        Let $K$ be a field of type (a) or (b) as above.  Let $E$ be a finite $K$-group that fits in a split exact sequence of finite $K$-groups
        \[
            1 \to A \to E \to F \to 1
        \]
        where $A$ is abelian and $F$ has approximation away from $\Bad_F$. Then, $E$ has approximation away from $\Bad_E$. 
    \end{thmx}  

    The proof of this Theorem follows the lines of a similar result for number fields settled in \cite[Theorem 4.1]{Lu-Nef-Dem}. 

    \subsection*{Organization of the article}
        In section \ref{Sec implicit function} we state one of the main ingredients to prove our main result, the implicit function theorem. For fields of type (a) this is a result that can be found in the literature. Most of this section is devoted to proving a variant of the implicit function theorem for certain schemes over a field that are not of finite-type (Proposition \ref{prop etale are local homeo non finite type}). This variant is needed to prove the Reocurrence Theorem for fields of type (b). \par 
        Section \ref{sec reocurrence} is devoted to the proof our main result (Theorem \ref{thm reocurrence}) corresponding to Theorem \ref{thm reocurrence intro} in the introduction. In subsection \ref{sec finer reocurrence} we give a refinement of Theorem \ref{thm reocurrence} that will be useful for applications to weak approximation. \par 

        In section \ref{section very weak approximation} we adapt the arguments in \cite{Lu-Nef-Dem} to obtain Theorem \ref{thm weak approx split extension} which corresponds to Theorem \ref{thm weak approx intro} in the introduction. \par

        Finally, in the appendix we give a proof of the analogous statement to Theorem \ref{thm reocurrence} for finite field extensions of $\C(x,y)^h$ (Theorem \ref{thm reocurrence henselian}), i.e. the fraction field of the henselianization of $\C[x,y]$ at the origin. Even though it is possible to adapt the methods used for finite field extensions of $\C(\! (x,y)\!)$ to prove Theorem \ref{thm reocurrence henselian}, this alternative proof avoids working with non-finite type schemes and builds upon the case of function fields of curves above $\C(t)^h$.
    \subsection*{Notation and conventions}
        \noindent\textbf{Fields and Galois cohomology:} \hspace*{1ex} Let $K$ be a field. We denote by $\Gamma_K$ the absolute Galois group of $K$ and by $\mu$ the $\Gamma_K$-module of roots of unity in an algebraic closure of $K$. By a finite $K$-group we mean a finite group $F$ together with a continuous action of $\Gamma_K$. For a finite abelian $K$-group $F$, we denote by $\H^i(K,F)$ its Galois cohomology. When $F$ is not necessarily abelian, we still denote by $\H^i(K,F)$ its non-abelian cohomology sets for $i=0,1$ as defined in \cite[Chapter 3]{serre1979galois}. We say that a finite field extension $L/K$ \textit{splits} a $K$-group $F$ if the absolute Galois group of $L$ acts trivially on $F$.  \vspace{1ex}\par
    \noindent\textbf{Schemes:} \hspace*{1ex}Let $K$ be a field. By $K$-variety we mean a separated $K$-scheme of finite type. Given a $K$-scheme $X$, we denote the residue field of $X$ at some fixed point $x\in X$ by $k(x)$. Moreover, for a field extension $L/K$, we denote $X \otimes_K L$ the base change of $X$ to $L$. Finally, given $L/K$ a finite field extension and $Y$ an $L$-variety, we denote by $\mathrm{R}_{L/K}Y$ the Weil restriction of $Y$ to $K$.
    
    \subsection*{Acknowledgements}
    I am deeply thankful to Diego Izquierdo, my PhD advisor, for all the explanations, observations and support. I would also like to sincerely thank Giancarlo Lucchini Arteche for giving me the opportunity to work in such an interesting area and for all the help he gave me. Writing this article would not have been possible if it were not for Diego and Giancarlo.

    \section{An implicit function theorem over Laurent series fields}\label{Sec implicit function}
    
    Let $k$ be a Hausdorff topological field. Recall that, as in \cite[app.III]{Weil1962foundations} or \cite[p. 256]{KS1983UnramifiedClassField}, one can construct a topology on the set of rational points of a $k$-variety $S$ as follows:
    \begin{enumerate}
        \item Choose an affine open cover $\{U_i\}_{i =1}^n$ of $S$, so that $S(k) = \bigcup_{i=1}^n U_i(k)$.
        \item Every section $f \in \Gamma(U_i,\mathcal{O}_{S})$ induces a function $f_*:U_i(k) \to k$. We endow $U_i(k)$ with the initial topology with respect to all functions induced by elements of $\Gamma(U_i,\mathcal{O}_S)$. This is the coarsest topology such that every function induced by some element in $\Gamma(U_i,\mathcal{O}_{S})$ is continuous.
        \item One can check that these topologies agree on the intersections of the $U_i(k)$ and, hence, one can glue them to get a topology on $S(k)$.
    \end{enumerate}
    Let $\ell/k$ be a finite extension of topological fields, meaning that $\ell$ is a Hausdorff topological field such that $k$ is closed in $\ell$ and the topology of $k$ is induced by that of $\ell$. For a $k$-variety $S$, we will denote by $S(\ell)$ the previous construction applied to $S \otimes_k \ell$ viewed as an $\ell$-variety. This construction turns out to have many nice properties. Here are some of them that will be used later in the article:
\begin{enumerate}
        \item[(T1)] \label{T1} For every morphism $f:X \to Y$ of $k$-schemes, the induced map $f_k: X(k) \to Y(k)$ is continuous.
        \item[(T2)] \label{T2} If $i: Z \hookrightarrow X$ is a closed (resp. open) immersion, the induced map $i_k: Z(k) \to X(k)$ is also a closed (resp. open) immersion. 
        \item[(T3)] \label{T3} If $\ell$ is a finite extension of the topological field $k$, then $S(k)$ is closed in $S(\ell)$ and its topology is induced by that of $S(\ell)$.
    \end{enumerate}
    Properties (T1) and (T2) can be found in \cite[Section 3]{GGM-B2014fibrepincipaux}. Since we could not find a reference for (T3) we will give a short proof.
    \begin{proof}[Proof of (T3)]
        First of all note that $S(k)$ is naturally a subset of $S(\ell)$. Since being closed is a local property, we may assume without loss of generality that $S = \Spec A$ is an affine scheme. Let us prove that the topology of $S(k)$ is induced by the one of $S(\ell)$. In order to do that, it is enough to prove that for every $g \in A\otimes_{k} \ell$ the restriction of the induced function $g_*:S(\ell) \to \ell$ to $S(k)$ is continuous. Indeed, it is easy to check that the subspace topology is the initial topology with respect to these functions. Every element $g \in A\otimes_k \ell$ can be written as $g = \sum_{i=1}^n f_i \otimes \alpha_i$ with $f_i \in A$ and $\alpha_i \in \ell$. Note that the functions $f_{i*}:S(k) \to k$ induced by $f_i$ are continuous for every $i$ by definition of the topology of $S(k)$ and that for every $x \in S(k)$ the evaluation of $g_*$ in $x$ is given by $g_*(x) = \sum_{i=1}^n f_{i*}(x) \alpha_i$ . Then, by continuity of the multiplication and sum of $\ell$, we can deduce that the restriction of $g$ to $S(k)$ is continuous. \par
        It remains to prove that $S(k)$ is closed. This can be done by checking the following equality
        \[
            S(k) = \bigcap_{f \in A} (f \otimes 1)^{-1}_*(k).
        \]
        The inclusion $S(k) \subseteq \bigcap_{f \in A} (f \otimes 1)^{-1}_*(k)$ is trivial. Let $x \in \bigcap_{f \in A} (f \otimes 1)^{-1}_*(k)$. Since $S$ is affine, $x$ corresponds to a morphism of $\ell$-algebras $x^*:A\otimes_k \ell \to \ell$. Moreover, the image of $A$ by $x^*$ is contained in $k$ because $x$ is in $\bigcap_{f \in A} (f \otimes 1)^{-1}_*(k)$,. We conclude that $x$ lies in $S(k)$.
    \end{proof}
    
    A final important property of the previous topology is the implicit function theorem, which we recall here:

    \begin{prop} \label{prop val henselian is top henselian}
        Let $k$ be a valued henselian field and $f:Y \to X$ an \'etale morphism of varieties over $k$. Then the continuous map $f_k: Y(k) \to X(k)$ is a local homeomorphism.
    \end{prop}
    
    A proof of this statement can be found in \cite[Proposition 3.1.4]{GGM-B2014fibrepincipaux}. \\
    
    In the sequel, we will need to construct similar topologies with analogous properties for some $k$-schemes that are not of finite type. To do so, let us fix $S$ a Noetherian $k$-scheme and define its set of rational points $S(k)$ as the set of sections of the structural morphism $S \rightarrow \mathrm{Spec}\, k$. By proceeding exactly in the same way as before, one can endow $S(k)$ with a topology satisfying properties (T1) to (T3). \\
   
     As for the implicit function theorem, we are going to prove it in a particular case that will be useful for us later in the article. We consider the following situation. Let $k_0$ be a perfect field. Set $k:=k_0(\! (x)\! )$ and 
     \[
     \hat{\mathbf{A}}^n := \Spec k_0[\![x,\underline{z}]\!][x^{-1}].
     \]
     Note that $\hat{\mathbf{A}}^n$ is naturally a $k$-scheme.
      \begin{lemma}\label{lemma points of hatA}
        Let $\ell / k$ be a finite field extension. Denote by $\mathfrak{m}_{\ell}$ the maximal ideal of the valuation ring of $\ell$. Then there is a natural homeomorphism $\hat{\mathbf{A}}^n(\ell) = \mathfrak{m}_{\ell}^n$.
    \end{lemma}
    \begin{proof}
       We may assume without loss of generality that $\ell = k$. Denote by $\mathfrak{m}$ the maximal ideal of $k_0[\![x]\!]$. Note that an element of $\hat{\mathbf{A}}^n(k)$ corresponds to a morphism of $k_0[\![x]\!]$-algebras $k_0 [\! [x, \underline{z}]\!][x^{-1}] \to k$. Since $x$ is invertible in $k = k_0(\!(x)\!)$, such a morphism is the same information as a morphism of $k_0[\![x]\!]$-algebras $\varphi :k_0 [\! [x, \underline{z}]\!] \to k$. We claim that such a morphism is the same data as $n$ elements $f_1,\dots,f_n \in \mathfrak{m}$. To do so it is enough to  prove that $f_i := \varphi(z_i)$ is in $\mathfrak{m}$ for each $i$ and that the morphism is continuous, in particular, it is determined by the images of the $z_i$. \par 

       Let $p$ be an element of the maximal ideal of $k[\![x,\underline{z}]\!]$, that is $(x,\underline{z})$, and set $f:= \varphi(p)$. Suppose that $f \not \in \mathfrak{m}$ and consider the series
        \[
        g= \sum_{j\geq 0} (\frac{1}{f} p)^j.
        \]
        The series $g$ is in $k_0[\![x,\underline{z}]\!]$ because $\frac{1}{f}\in k_0[\![x]\!]$ and satisfies $g = 1 + \frac{p}{f} g$. This implies that $\varphi(g) = 1 + \varphi(g)$ which is a contradiction. We deduce that the image of $\varphi$ is in $k_0[\! [x]\!]$. Moreover, $\varphi$ is continuous because for every $r\in \N$ it maps the $r$-th power of the ideal $(x,\underline{z})$ into the $r$-th power of $\mathfrak{m}$. Thus the morphism $\varphi$ is determined by a choice of image for each $z_i$ and every choice of images $f_1,\dots,f_n \in \mathfrak{m}$ defines a morphism. \par 

        Up until now we have established a bijection $g:\hat{\mathbf{A}}^n(k) \to \mathfrak{m}^n$. It remains to prove that $g$ and $g^{-1}$ are continuous. To prove that $g$ is continuous it is enough to check that the composition of $g$ with the $i$-th projection $\pi_i:\mathfrak{m}^n \to \mathfrak{m}$ is continuous. The composition $\pi_i \circ g$ is indeed continuous because it is the function induced by the global section $z_i \in k_0[\![x,\underline{z} ]\!][x^{-1}]$. \par 
        
        Since the topology of $\hat{\mathbf{A}}^n(k)$ is initial with respect to the functions induced by every element of $k_0[\![x,\underline{z} ]\!][x^{-1}]$, in order to prove that $g^{-1}$ is continuous it is enough to prove that for every $p \in k_0[\![x,\underline{z} ]\!][x^{-1}]$ the composition $p_* \circ g^{-1}$ is continuous. The function $p_* \circ g ^{-1}$ is continuous because it is the evaluation of the series $p$, that is to say for every $a_1,\dots ,a_n \in \mathfrak{m}$ we have $p_*\circ g^{-1}(a_1,\dots ,a _n) = p(x,a_1,\dots ,a_n)$. We conclude because evaluation of a formal series is continuous with respect to the topology of $k$.
    \end{proof}
    \begin{prop}\label{prop implicit function hatA}
        Let $X$ and $Y$ be $k$-schemes and $f: Y \to X$ a morphism. Assume $X$ is a closed sub-$k$-scheme of $\hat{\mathbf{A}}^n$ and that $f$ is \'etale at a point $v \in Y$ such that $k(v) = k$. Then the map $f_k: Y(k) \to X(k)$ is a local homeomorphism at $v$.
    \end{prop}
    \begin{proof}
        Thanks to \cite[\S 2.3 Proposition 3]{NeronModels} we can assume that $X = \Spec A$ is affine and $Y = \Spec A[T]/P(T)$ where $P$ is a polynomial such that $P'(v) \neq 0$. Let $\Tilde{P}$ be a lifting of $P$ to $k_0[\![x, \underline{z}]\!][x^{-1}]$. By property (T2) we can reduce the proof to the case where $X = \hat{\mathbf{A}}^n$ and $Y = \Spec k_0[\![x, \underline{z}]\!][x^{-1},T]/\Tilde{P}$. Furthermore, we may translate the closed immersion $Y \hookrightarrow \hat{\mathbf{A}}^{n+1}$ to assume that $v$ corresponds to the ideal $(\underline{z},T)$ and its image to $(\underline{z})$. Since $f$ is \'etale at $v$ there is a suitable choice of coordinates such that $\Tilde{P}$ is of the form
        \[
            \Tilde{P}(\underline{z},T) = T + \sum_{|I|+j \geq 2}a_{I,j}\underline{z}^I T^j \in k_0[\![x,\underline{z}]\! ][x^{-1},T].
        \]
        After replacing $\Tilde{P}(\underline{z},T)$ by $\frac{1}{\alpha}\Tilde{P}(\alpha \underline{z},\alpha T)$ for $\alpha \in k$ small enough we can assume that $a_{I,j} \in \mathfrak{m}$ for every $I$ and $j$. Note that for every $\underline{z_0} \in \mathfrak{m}$ the polynomial $\Tilde{P}(\underline{z_0},T)$ has a unique root in $\mathfrak{m}$ by Hensel's lemma. We can rephrase this as $f_{k}$ induces a bijection between $Y(k)$ and $\widehat{\mathbf{A}}^n(k)$. It remains to prove that its inverse is continuous at the origin. The same argument as in the proof of \cite[Proposition 3.1.4]{GGM-B2014fibrepincipaux} works.
    \end{proof}
    \begin{prop}\label{prop etale are local homeo non finite type}
        Let $L/K/k_0 (\!( x,y )\! )$ be a tower of finite extensions. Denote by $A$ and $B$ the integral closures of $k_0[\! [x,y]\!]$ in $K$ and $L$ respectively. Set $X := \Spec A[x^{-1}]$ and $Y := \Spec B[x^{-1}]$. Let $v \in Y$ be such that the natural $k$-morphism $f:Y \to X$ is \'etale at $v$. Then $f$ induces a map $f_v: Y(k(v)) \to X(k(v))$ that is a local homeomorphism at $v$. 
    \end{prop}
    \begin{proof}
       We may assume that $k(v) =k$ after a suitable base change. In order to apply Proposition \ref{prop implicit function hatA} we only need to check that $X$ admits a closed immersion to $\hat{\mathbf{A}}^n$ for certain $n$. This can be done thanks to \cite[Theorem 28.3]{Matsumura1989CommutaiveRing}.
    \end{proof}
    Lastly we would like prove the following fact.
    \begin{prop} \label{prop opens are infinite}
        Keep the notation from the previous proposition. Every non-empty open subset of $X(k)$ is infinite.
    \end{prop}
    \begin{proof}   
        Let $U$ be a non-empty open set of $X(k)$. For a point $v$ in $U$ take $\pi_v \in A$ an uniformizer at $v$. We may define a morphism $k_0[\![x,y]\!][x^{-1}] \to A[x^{-1}]$ above $k$ by sending $y$ to $\pi_v$. The corresponding morphism of schemes $X \to \hat{\mathbf{A}}^1$ is \'etale at $v$. Then the morphism $\pi_{v*}: U \to \mathfrak{m}$ is a local homeomorphism at $v$ by Proposition \ref{prop etale are local homeo non finite type}. In particular, any non-empty open subset of $X(k)$ contains an open subset homeomorphic to an open set of $\mathfrak{m}$. We conclude by observing that every non-empty open subset of $\mathfrak{m} = xk_0 [\![x]\!]$ is infinite.
    \end{proof}
\section{Reocurrence}\label{sec reocurrence}

    In this section we prove the main theorem of the article (Theorem \ref{thm reocurrence}) and a refinement (Theorem \ref{thm finer reocurrence}) which is necessary for some applications of our result to very weak approximation (Theorem \ref{thm weak approx split extension}). \par 
    
    During this section $K$ will be a field of one of the following types:
    \begin{enumerate}
        \item[(a)] the function field of a smooth projective curve $C$ over a henselian valued field $k$.
        \item[(b)] a finite extension of $k_0(\!( x,y)\! )$ where $k_0$ is a perfect field.
    \end{enumerate}
    For type (a) we consider the set of places of $K$ coming from closed points of the curve $C$, which we denote by $C^{(1)}$. For fields of type (b) we denote by $A$ the integral closure of $k_0[\! [x,y] \!]$ in $K$ and by $C$ the scheme $\Spec A[x^{-1}]$. For type (b) we fix $k$ to be $k_0(\! ( x )\!)$. In both cases we refer to $C$ as the Dedekind scheme associated to $K$. \\

        \begin{thm} \label{thm reocurrence}
            Let $L/K$ be a finite Galois extension with group $G$ and $v_0 \in C^{(1)}$ a place that is unramified in $L/K$. Then there exists an open neighbourhood $U$ of $v_0$ in $C(k(v_0))$ such that for every $v \in U$ the decomposition groups $G_v$ and $G_{v_0}$ are conjugates in $G$.
        \end{thm}
    
    The only justification to invert $x$ in type (b) is for $C$ to be naturally a $k$-scheme.
    \begin{rmk}
        Theorem \ref{thm reocurrence intro} can be deduced from Theorem \ref{thm reocurrence}. Indeed, if $v_0$ is a place above $x$, it would be enough to exchange the roles of $x$ and $y$ in the construction of the topology. Moreover, the open set in the statement of Theorem \ref{thm reocurrence} is necessarily uncountable because for type (a) it contains an open subset of $k$ and for type (b) it contains an open subset of $xk_0[\![x]\!]$. Furthermore, if $K$ is of type (b), the cardinality of the open subset is strictly bigger than the cardinality of $k_0$.
    \end{rmk}
        Before proving Theorem \ref{thm reocurrence} we will prove the following key proposition.
       
    \begin{prop}\label{prop degree of place is open}
        Let $L/K$ be a finite extension and $v \in C^{(1)}$ a place that is unramified in $L/K$. Denote by $D$ the Dedekind scheme associated to $L$. Let $w \in D^{(1)}$ be a place of $L$ above $v$. Then there exists a locally closed subset $W$ of $D(k(w))$ containing $w$ and satisfying the following properties
        \begin{enumerate}
            \item Its image $V$ in $C(k(w))$ lies in $C(k(v))$.
            \item $V$ is open in $C(k(v))$.
            \item For every $w' \in W$ the degree of $w'$ in $L/K$ is the same as the degree of $w$.
        \end{enumerate}
    \end{prop}
    \begin{proof}
        Consider the following cartesian diagram
        \begin{equation*}
            \begin{tikzcd}
                \phi^{-1}(C_{k(v)}) \ar[r,"\psi"] \ar[d] & C_{k(v)} \ar[d] \\
                 \mathrm{R}_{k(w)/k(v)} D_{k(w)} \setminus \; \bigcup_{k(v)\subseteq m \subsetneq k(w) } \mathrm{R}_{m/k(v)}D_m \ar[r,"\phi"] & \mathrm{R}_{k(w) / k(v)} C_{k(w)}.
            \end{tikzcd}
        \end{equation*}
        We may apply Proposition \ref{prop val henselian is top henselian} for type (a) or Proposition \ref{prop etale are local homeo non finite type} for type (b) to find open neighbourhoods $V$ of $v$ in $C(k(v))$ and $W$ of $w$ in $\phi^{-1}(C_{k(v)})(k(v))$ such that $\psi_{k(v)}$ induces a homeomorphism between $W$ and $U$. Clearly $W$ and $V$ satisfy conditions 1 and 2 because $\phi^{-1}(C_{k(v)})(k(v))$ is a closed subset of $D(k(w))$. Finally, after removing the elements of $W$ that ramifies in $L/K$, we see that $W$ satisfies condition 3 because it is a closed subset of $\bigcup_{k(v) \subseteq m \subsetneq k(w)}D(m)$ and therefore the residue field of every $w' \in W$ is equal to $k(w)$.
    \end{proof}
    Proposition \ref{prop degree of place is open} has as corollaries the following lemmas.
    \begin{lemma} \label{lemma reocurrence degree one places}
            Let $L/K$ be a finite extension. Let $\ell/k$ be a finite field extension. Then the set of places $v \in C(\ell)$ such that $L$ has a place of degree one over $v$ and $v$ is unramified in $L/K$ is open.
    \end{lemma}
     \begin{lemma} \label{lemma reocurrence inert}
     Let $L/K$ be a finite Galois extension with group $G$ and $v_0$ a place of $K$ inert in $L/K$. Then, there exists an open neighbourhood $U \subseteq C(k(v_0))$ of $v_0$ such that every $v \in U$ is inert.
    \end{lemma}
    Indeed, to prove Lemma \ref{lemma reocurrence degree one places} we only need to apply Proposition \ref{prop degree of place is open} to a place of degree one and for Lemma \ref{lemma reocurrence inert} to a place of degree $[L:K]$.
    \begin{proof}[Proof of Theorem \ref{thm reocurrence}]
        Denote by $D$ the Dedekind scheme corresponding to $L$. Fix $w_0$ a place of $L$ in $D^{(1)}$ above $v_0$ to define the decomposition group $G_{v_0}$. Let $L'$ be the fixed field of the decomposition group $G_{v_0}$ and $D'$ the Dedekind scheme corresponding to $L'$. Denote by $v_0'$ the restriction of $w_0$ to $L'$. By definition of $L'$, the extension $L/L'$ is inert at $v_0'$ and $v_0'$ has degree one above $v_0$, in particular, $k(v_0') = k(v_0)$. Then we can apply Lemma \ref{lemma reocurrence inert} to get an open neighbourhood $U'$ of $v_0'$ in $D'(k(v_0))$ such that every element of $U'$ is inert in $L/L'$. On the other hand, thanks to Lemma \ref{lemma reocurrence degree one places}, there is an open neighbourhood $V$ of $v_0$ in $C(k(v_0))$ such that every element of $V$ has a place of $L'$ of degree one above . \par 
        
        The image of $U'$ in $C(k(v_0))$ is an open set because every element of $U'$ is unramified and we can apply Proposition \ref{prop val henselian is top henselian} or Proposition \ref{prop etale are local homeo non finite type}. Let us denote by $U$ the intersection of this image with $V$ in $C(k(v_0))$. It is easy to see that the decomposition group of every element of $U$ is conjugate to $G_{v_0}$.
    \end{proof}
    As mentioned before, this result can be understood as a ``reocurrence'' phenomenon. In general, we cannot prove existence of places with a given decomposition group, but for totally split places we can.
    \begin{prop}\label{prop existence totally split}
        Let $L/K$ be a finite Galois extension. Then there exists a place $v \in C^{(1)}$ of $K$ that is totally split in $L/K$. 
    \end{prop}
    \begin{proof}
        Denote by $D$ the Dedekind scheme corresponding to $L$ and by $\varphi: D \to C$ the morphism corresponding to the extension $L/K$. Let $\ell/k$ be the smallest extension such that $D(\ell)\neq \emptyset$. A place $v \in C(\ell)$ is totally split when it is unramified, there is a place $w \in D(\ell)$ above $v$, and $ v \not \in \bigcup_{k \subseteq m \subsetneq \ell} C(m)$. \par 
        We can apply Proposition \ref{prop val henselian is top henselian} or Proposition \ref{prop etale are local homeo non finite type} to deduce that $\varphi: D \to C$ induces a local homeomorphism $\varphi_{\ell}: D(\ell) \to C(\ell)$ at the unramified points. In particular, for every non-empty open set $U$ of $D(\ell)$ we have $\varphi_{\ell}(U) \not \subseteq \bigcup_{k \subseteq m \subsetneq \ell} C(m)$. Therefore, there is a place $w \in U$ such that $v:=\varphi(w)$ is totally split in $L/K$. \par 
    \end{proof}
    In fact, the previous argument proves that the set of places $w \in D(\ell)$ above totally split places is dense in $D(\ell)$. \\ 
    
    One may think that Theorem \ref{thm reocurrence} has as a direct corollary that a finite Galois extension $L/K$ is either trivial at every place in $C^{(1)}$ or it has infinitely many places where it is non trivial. This is not the case because an extension might be non-trivial only at the ramified places. We can restate this in terms of the Tate-Shafarevich groups: 
    
    \begin{defi}\label{def Tate-Shafarevich groups}
        Let $G$ be an abelian algebraic $K$-group, $S$ a subset of $C^{(1)}$ and $n \in \N$. We define the $n$-th Tate-Shafarevich of $G$ as
         \[
        \Sha^i_{S}(K,G):= \ker \left( \H^i(K,G) \to \prod_{v \in C^{(1)} \setminus S} \H^i(K_v,G)\right).
        \]
        Also fix $\Sha^i(K,G):= \Sha^i_{\emptyset}(K,G)$ and 
        \[
        \Sha^i_{\omega}(K,G) = \bigcup_{\substack{S \subseteq C^{(1)}\\ \text{finite}}} \Sha_S^i(K,G).
        \]
    \end{defi}
    
    Then Theorem \ref{thm reocurrence} does not imply directly that $\Sha^1(K,\Z/n \Z) = \Sha_{\omega}^1(K,\Z/n \Z)$. When $K$ is a function field of a curve over $\C(\!(t)\!)$, this result can be found in \cite[Proposition 2.6 (ii)]{CTH2015DualCt}. We claim that the same proof works if $K$ is a function field of a curve over a strictly henselian or a finite extension of $\C(\!(x,y)\!)$. We replicate the argument for convenience of the reader.
    \begin{prop} \label{prop Shaomega equal Sha}
       Assume that $K$ is a function field of a curve defined over a strictly henselian $k$ or a finite extension of $\C(\!(x,y)\!)$. Let $n$ be an integer prime to the characteristic of $k$. Then 
        \[
        \Sha_{\omega}^1(K,\Z/n\Z) = \Sha^1(K,\Z/n\Z).
        \]
    \end{prop}
    \begin{proof}
        Since $K$ contains the $n$-th roots of unity, we have $\H^1(K,\Z/n\Z) = K^{*}/K^{*n}$ and $\H^1(K_v,\Z/n\Z) = K_v^* / K_v^{*n}$ for every $v \in C^{(1)}$. Let $f \in K^*$ be such that $[f]\in \Sha_{\omega}^1(K,\Z/n\Z)$. Let $v\in C^{(1)}$ be such that $[f_v] \neq 0 \in \H^1(K_v,\Z/n\Z)$. We consider two cases:
        \begin{itemize}
            \item If the extension $K(\sqrt[n]{f})/K$ does not ramify at $v$, Theorem \ref{thm reocurrence} contradicts the fact that $f \in \Sha_{\omega}^1(K,\Z/n\Z)$. In the case when $K$ is a finite field extension we need to use Proposition \ref{prop opens are infinite}.
            \item Suppose that the extension $K(\sqrt[n]{f})/K$ ramifies at $v$. Fix a uniformizer $\pi$ of $K$ at $v$. There exist a unit $u$ at $v$ and $r \in \Z$ not divisible by $n$ such that $f = u\pi ^r$. The function $\pi: C(k(v)) \to k(v)$ induced by $\pi$ is a local homeomorphism at $v$ thanks to Proposition \ref{prop val henselian is top henselian} in the case of a function field and Proposition \ref{prop etale are local homeo non finite type} for finite field extensions of $k_0(\!(x,y)\!)$. Which means that there is an open neighbourhood $U$ of $v$ in $C(k(v))$ and $V$ of $0$ in $k(v)$ such that $\pi$ induces a homeomorphism between $U$ and $V$. We may assume that for every $q \in U$ the class of $u(q) \in k(v)^* / k(v)^{*n}$ is constant. On the other hand, as $q \in U$ varies the value $\pi(q)$ goes through all the classes of $k(v)^*/k(v)^{*n}$. Hence there are infinitely many points $q \in C(k(v))$ such that $f$ is invertible at $q$ and $f(q)$ is not a $n$-th root in $k(v)$. This gives an infinite number of places $w$ in $C^{(1)}$ such that $f_w \neq 0 \in H^1(K_w,\Z/n\Z)$. We conclude because this contradicts that $f$ is in $\Sha_{\omega}^1(K,\Z/n\Z)$.
        \end{itemize}
    \end{proof}

    \subsection{Finer reocurrence} \label{sec finer reocurrence}

    In this section we discuss a finer version of the Reocurrence Theorem that gives a good replacement for Frobenius elements. In order to do so we need to introduce a technical condition.
    
    \begin{defi} \label{def many automorphisms}
        Let $\Gamma$ be a profinite group. We say that $\Gamma$ has enough automorphisms if for every finite group $G$, automorphism $\alpha$ of $G$ and pair of surjective and continuous morphisms $f,g: \Gamma \to G$ there exists an automorphism $\beta$ of $\Gamma$ that makes the following diagram commute
        \begin{equation*}
            \begin{tikzcd}
                \Gamma \ar[d,"f"] \ar[r,"\beta"] & \Gamma \ar[d,"g"] \\
                G \ar[r,"\alpha"] & G.
            \end{tikzcd}
        \end{equation*}
    \end{defi}
    Note that in the previous definition we can always take $\alpha$ to be the identity. Indeed, if we can lift the identity, we can lift $\alpha$ by replacing $f$ by $\alpha \circ f$. Let us prove that this definition is not empty.
    \begin{prop} \label{prop hatZ has enough automorphism}
        The profinite group $\widehat{\Z}^n$ has enough automorphisms.
    \end{prop}
    \begin{proof}
        Let $p$ be a prime number. Before proving the result for $\widehat{\Z}^n$, let us prove it for $\Z_p^n$. Fix a finite group $G$ and two surjective and continuous morphisms $f,g:\Z_p^n \to G$. As mentioned before, we only need to lift the identity, that is we assume $\alpha = \id_G$. Since $\Z_p$ is a principal ideal domain, we can find a base $e_1,\dots, e_n$ of $\Z_p^n$ such that the morphism $f$ is the quotient by a subgroup of the form $p^{m_1} \Z_p \oplus \dots \oplus p^{m_n} \Z_p$. We can find another basis $e_1',\dots, e_n'$ such that $g$ is of the same form. We can define the automorphism $\beta$ that we are looking for by sending $e_i$ to $e_i'$ for every $i \in \{1,\dots,n\}$. \par 
        Now, let $G$ be a finite group and $f,g: \widehat{\Z}^n \to G$ two surjective and continuous morphisms. We can decompose $G$ in its primary components and $\hat{\Z}$ as the product $\prod_{p \text{ prime}} \Z_p$, which allow us to work one prime at a time. That is to say, we have commutative triangles
        \begin{equation*}
            \begin{tikzcd}[column sep =2ex]
                \Z_p^n \ar[dr,swap,"f|_{\Z_p}"] \ar[rr,"\simeq"] & &\Z_p^n \ar[dl,"g|_{\Z_p}"]\\
                  & G\{ p \} &.
            \end{tikzcd}
        \end{equation*}
        where $G\{p\}$ is the $p$-primary torsion subgroup of $G$. This data glues to give the desired automorphism of $\widehat{\Z}^n$.
    \end{proof}
    \begin{thm} \label{thm finer reocurrence}
        Let $K$ be the function field of a curve over a valued henselian field $k$ whose absolute Galois group has enough automorphisms or a finite extension of $k_0 (\! (x,y)\! )$ such that $\Gamma_{k_0}\times \hat{\Z}$ has enough automorphisms. Let $L/K$ be a finite Galois extension and $v \in C^{(1)}$ unramified in $L/K$.  Then there is an open neighbourhood $U$ of $v$ in $C(k(v))$ such that for every $w \in U$ we have an isomorphism $\phi_{v,w}:\Gamma_{K_{v_0}} \to \Gamma_{K_w}$ together with an element $\rho_w \in \Gamma_K$ such that for every $s \in \Gamma_{K_{v_0}}$, the images of $\phi_{v_0,w}(s)$ and $\rho_w s \rho_w^{-1}$ in $G$ are the same.
    \end{thm}
    \begin{proof}
       We can apply Theorem \ref{thm reocurrence} to get an open neighbourhood $U$ of $v$ in $C(k(v))$ such that the for every $w \in U$ the decomposition groups of $v$ and $w$ are conjugates. If we shrink $U$ we can assume that the residue fields at $v$ and $w$ are isomorphic. For $w \in U$ take $\rho_w \in G$ such that $G_{v} = \rho_w G_{w}\rho_{w}^{-1}$. Let $f: \Gamma_K \to G$ be the surjective morphism corresponding to the extension $L/K$ and $f_v:\Gamma_{K_v} \to G$ its restriction to $\Gamma_{K_v}$, similarly for $w$. Then the question is reduced to finding an isomorphism $\phi:\Gamma_{K_v} \to \Gamma_{K_w}$ making the following diagram commute 
       \begin{equation*}
           \begin{tikzcd}
               \Gamma_{K_v}\ar[d] \ar[r,"\phi"] & \Gamma_{K_w} \ar[d] \\
               G_v \ar[r,"c_{\rho_w}"] & G_w
           \end{tikzcd}
       \end{equation*}
       where $c_{\rho_w}$ is conjugation by $\rho_{w}$. Indeed, if we are able to find such an isomorphism, we would have $\rho_{w}f_v(s)\rho_w^{-1} = f_w(\phi(s))$ and lifting $\rho_{w}$ to $\Gamma_K$ would give the result. \par
       Now, note that if $K$ is of type (a), we have $\Gamma_{K_v} \simeq \Gamma_k \times \widehat{\Z}$ and the vertical arrows in the previous diagram are still surjective after restricting to the factor $\Gamma_k$ because the extension is unramified at $v$ and $w$. So we can conclude, for type (a), using that $\Gamma_k$ has enough automorphisms. For type (b) we may use the same argument after noticing that the local fields have absolute Galois group isomorphic to $\Gamma_k \times \hat{\Z}^2$.
   \end{proof}
   In particular, by Proposition \ref{prop hatZ has enough automorphism} we can apply the theorem to function fields of curves over $\C(\!(t_1)\!) \cdots (\!(t_n)\!)$ or finite field extensions of $\C(\!(t_1)\!) \cdots (\!(t_n)\!)(\!(x,y)\!)$. We do not know if Theorem \ref{thm finer reocurrence} is valid for function fields over $p$-adic fields. It would be interesting to determine whether the absolute Galois group of a $p$-adic field has enough automorphisms or to find an example of a non-commutative absolute Galois group that has enough automorphisms.

    \section{Very weak approximation} \label{section very weak approximation}

        In this section we present an application of Theorem \ref{thm reocurrence} to very weak approximation over some function fields and some Laurent series fields. We follow the the proof of \cite[Theorem 4.1]{Lu-Nef-Dem} and adapt it to our setting. There are new difficulties to overcome because Chebotarev's density Theorem must be replaced by Theorem \ref{thm reocurrence} or by Theorem \ref{thm finer reocurrence}. \par
        
        Let $\C$ be an algebraically closed field of characteristic zero. In this section $K$ is a field of one of the following types 
        \begin{enumerate}
            \item[(a')] A function field of a curve over $\C(\!(t)\!)$;
            \item[(b')] A finite field extension of $\C(\!(x,y)\!)$.
        \end{enumerate} 
        Similarly to previous section, for type (a') we denote by $C$ the curve whose function field is $K$. In case (b') we change the notation slightly, $C$ denotes the scheme $\Spec A \setminus \mathfrak{m}$ where $A$ is the integral closure of $\C[\![x,y]\!]$ in $K$ and $\mathfrak{m}$ is the maximal ideal of $A$. Hoping it does not lead to confusion, we still call $C$ the Dedekind scheme associated to $K$.

    \subsection{Reocurrence and cohomology}
        We begin by adapting the results in \cite[Section 3]{Lu-Nef-Dem} to our context. Let us start with a basic lemma.
        \begin{lemma}
        Let $\widetilde{K}$ be the compositum of all finite extensions $L/K$ that are totally split at every $v \in C^{(1)}$ in a fixed algebraic closure. The extension $\widetilde{K}/K$ is Galois.
        \end{lemma}
        \begin{proof}
            Let $x$ be an element of $\widetilde{K}$ and $p \in K[T]$ its minimal polynomial. By assumption the extension $K(x)$ is totally split at every $v \in C^{(1)}$. Then for every $v \in C^{(1)}$ the polynomial $p$ splits into linear factors in $K_{v}[T]$. This implies that the extension $K(y)$ where $y$ is a Galois conjugate of $x$ is totally split at every place $v$ in $C^{(1)}$. So we conclude that $\Tilde{K}/K$ is Galois.
        \end{proof}
        \begin{prop}\label{prop Sha in the image of inf}
            Let $A$ be a finite abelian $K$-group and $L/K$ a finite extension splitting $A$. For every finite subset $S \subseteq C^{(1)}$, the group $\Sha^1_{S}(K,A)$ is contained in the image of the inflation map from $\H^1(\widetilde{L}/K, A )$.
        \end{prop}
        \begin{proof}
            Let $D$ be the Dedekind scheme associated to $L$. Note that for every $w \in D^{(1)}$ we can define a restriction morphism $\H^1(\widetilde{L},A) \to \H^1(L_w,A)$. Indeed, the absolute Galois group of $L_w$ is a subgroup of the absolute Galois group of $\widetilde{L}$ because $\widetilde{L} \subseteq L_w$. We then consider the following commutative diagram with exact rows
            \begin{equation}\label{diag Inf res}
                \begin{tikzcd}
                    \H^1(\widetilde{L}/K,A) \ar[r,"\Inf"]& \H^1(K,A) \ar[r,"\Res"] \ar[d]& \H^1(\widetilde{L},A) \ar[d] \\
                     & \prod_{v \in C^{(1)}} \H^1(K_v,A) \ar[r,"\prod \Res_v"] &\prod_{v \in D^{(1)}} \H^1(L_w,A).
                \end{tikzcd}
            \end{equation}
        Note that $\H^1(\widetilde{L},A) = \mathrm{colim}_{M/L} \; \H^{1}(M,A)$ where the colimit is over all finite extensions $M/L$ that are trivial at every $v$ in $C^{(1)}$. Then for any class of $\beta \in \H^1(\widetilde{L},A)$ there is a finite Galois extension $M/L$ and $\beta' \in \H^1(M,A)$ representing $\beta$.
        Using Proposition \ref{prop Shaomega equal Sha}, we deduce that a class in $\H^1(\widetilde{L},A)$ that is trivial at all but finitely many places in $D^{(1)}$ is trivial at every place in $D^{(1)}$. \par 
        Let $\alpha$ be an element of $\Sha^1_{S}(K,A)$. The last paragraph proves that $\beta = \Res (\alpha)$ in $\H^1(\Tilde{L},A)$ is trivial at every place in $D^{(1)}$. Since $L$ splits $A$, we can think of $\beta$ as a morphism $\beta: \Gamma_{\widetilde{L}} \to A$. To this morphism corresponds a finite extension $M/\widetilde{L}$ trivial at every $w \in D^{(1)}$. This means that $\beta$ is trivial because every element of $M$ lies in a finite locally trivial extension of $L$ and $\widetilde{L}$ is the compositum of all those extensions. We conclude by exactness of the top row of diagram \eqref{diag Inf res} that $\alpha$ is in the image of the inflation map.
        \end{proof}
        \begin{defi} \label{def Bad places}
            Let $F$ be a finite $K$-group and $L/K$ the minimal extension splitting $F$. We denote by $\Bad_F$ the set of places where the extension $L/K$ ramifies.
        \end{defi}
         In contrast with \cite[Definition 2.1]{Lu-Nef-Dem}, we do not have to worry about the primes dividing the order of the group because our residue fields have characteristic zero.
        
        \begin{lemma}\label{lemma reocurence cohomology}
            Let $F$ be a finite $K$-group and $v \not \in \Bad_F$. Then there exists infinitely many places $w \in C^{(1)}$ together with isomorphisms
            \[
            \rho^*_{v,w}: \H^1(K_v,F) \to \H^1(K_w,F).
            \]
        \end{lemma}
        \begin{proof}
            Let $L/K$ be a Galois extension splitting the group $F$. Using Theorem \ref{thm finer reocurrence} we get infinitely many places $w$ in $C^{(1)}$ and elements $\rho_w \in \Gamma_K$ such that the image of $\phi_{v,w}(s)$ and $\rho_w s \rho_w^{-1}$ in $\Gal(L/K)$ are the same for every $s \in \Gamma_{K_v}$. \par 
            Denote by ${}_{\rho_w}F$ the $K$-group whose underlying group is the same as $F$, but the action is given by $s \cdot f := {}^{\rho_{w} s \rho^{-1}_w}f$. With this definition, the map from $F$ to ${}_{\rho_w}F$ given by $f \mapsto {}^{\rho_w}f$ is an isomorphism of $K$-groups, in particular, it induces a isomorphism of pointed sets $\rho_{w}^*: \H^1(K_v,F) \to \H^1(K_v,{}_{\rho_w}F)$.\par
            We can apply Theorem \ref{thm finer reocurrence} to the extension $L/K$ and get the desired isomorphisms as the following compositions 
            \begin{equation} \label{eq rhovw}
            \H^1(K_v,F) \xrightarrow{\rho_w^*} \H^1(K_v, {}_{\rho_{w}} F) \xrightarrow{\phi_{v,w}^*} \H^1(K_w,F).
            \end{equation}
            Note that $\phi_{v,w}$ induces a morphism on the cohomology sets because $\rho_w s \rho_w^{-1}$ and $\phi_{v,w}(s)$ have the same action on $F$.
        \end{proof}
        \begin{lemma} \label{lemma cocylces fun calculation}
            Let $F,v,w,\rho_{v,w}^*$ be as in Lemma \ref{lemma reocurence cohomology}. Let $L/K$ a finite extension splitting $F$. Take $\alpha \in \H^1(K_v,F)$ and suppose that there exists a class $\beta = [b] \in \H^1(K,F)$ such that $\beta_v = \alpha$ and $\beta_w = \rho_{v,w}^*\alpha$. Additionally suppose that the morphism $\Gamma_{L} \to F$ obtained by restricting $b$ is surjective. If we denote by $L'$ the extension corresponding to the kernel of $b|_{\Gamma_L}$, the decomposition groups of $v$ and $w$ are conjugate in $\Gal(L'/K)$
        \end{lemma}
        The following proof is a rephrasing of the proof of \cite[Lemma 3.3]{Lu-Nef-Dem}.
        \begin{proof}
            Note that the extension $L'/K$ is Galois thanks to \cite[Lemma 3.4]{Lu-Nef-Dem}. \par 
            There is an element $\rho_w \in \Gamma_K$ such that for every $s \in \Gamma_{K_v}$ the elements $\rho_w s \rho^{-1}_w$ and $\phi_{v,w}(s)$ have the same image in $\Gal (L/K)$. To simplify notation we will write $\rho$ instead of $\rho_w$. We begin the proof with the following claim\par
            \vspace*{1ex} \hspace*{2ex}\textbf{Claim:} We can choose $\rho$ such that $b_{\rho} = 1$  \vspace*{1ex} \newline
           For every $\chi \in \Gamma_L$ we have that $b_{\chi\rho} = b_{\chi} b_{\rho}$ because $\chi$ acts trivially on $F$. Since the restriction of $b$ to $\Gamma_L$ is surjective, there exists an element $\chi_0$ such that $b_{\chi_0} = b_{\rho}^{-1}$. Then after replacing $\rho$ by $\chi_0 \rho$ we have $b_{\rho} =1$ and the images of $\phi_{v,w}(s)$ and $\rho s \rho^{-1}$ in $\Gal(L/K)$ still coincide. \par 

           Now, according to the proof of Lemma \ref{lemma reocurence cohomology}, the morphism $\rho^*_{v,w}$ is given by the composition of $\rho_w^*$ and $\phi_{v,w}^*$. At the level of cocycles this means that for $\alpha =[a] \in H^1(K_v,F)$ the class $\rho_{v,w}^* \alpha$ is represented by the cocycle
           \[
                a'_{t} = {}^{\rho} a_{\phi_{v,w}^{-1}(t)}.
            \]
            Since $\beta_v = \alpha$ and $\beta_w = \rho_{v,w}^*\alpha$, there exists $f,f' \in F$ such that for every $s \in \Gamma_v$ we have
            \begin{equation*}
                {}^{\rho}(f b_{s} {}^{s} f^{-1}) = {}^{\rho} a_{s} = a_{\phi_{v,w}(s)}' = f' b_{\phi_{v,w}(s)}{}^{\phi_{v,w}(s)} f'^{-1}.
            \end{equation*}
            The restriction of $b$ to $\Gamma_L$ gives an isomorphism between $F$ and $\Gal(L'/L)$, hence there exist unique  $\chi$ and $\chi'$ in $\Gal(L'/L)$ such that $b_{\chi} = f$ and $b_{\chi'} = f'$. Therefore
            \begin{equation}\label{calculo cociclos}
            {}^{\rho}b_{\chi}{}^{\rho} b_{s} {}^{\rho s} b_{\chi}^{-1} = b_{\chi'} b_{\phi_{v,w}(s)} {}^{\phi_{v,w}(s)}b_{\chi'}^{-1}.
            \end{equation}
            Fix $\rho_0:=\chi'^{-1} \rho \chi$. We claim that the images of $\phi_{v,w}(s)$ and $\rho_0 s \rho_0^{-1}$ in $\Gal(L'/K)$ coincide for every $s\in \Gamma_{K_v}$. Since $\chi$ and $\chi'$ are elements of $\Gamma_L$, the images of $\rho_0 s \rho^{-1}_0$ and $\phi_{v,w}(s)$ in $\Gal(L/K)$ coincide. In particular, they have the same action on $F$. Note that $b_{\rho_0} = b_{\chi'}^{-1}{}^{\rho}b_{\chi}$ since $b_{\rho} =1$ and $\chi'$ acts trivially. Hence equality \eqref{calculo cociclos} gives 
            \[
            b_{\rho_0} {}^{\rho}b_{s} {}^{\phi_{v,w}(s)}b_{\rho_0}^{-1} = b_{\phi_{v,w}(s)}.
            \]
            The subgroup $\Gal(L'/L)$ of $\Gal(L'/K)$ is normal because the extension $L/K$ is Galois, so there exists $\psi \in \Gal(L'/L)$ such that $\rho_0 s \rho_0^{-1} = \psi \phi_{v,w}(s)$ in $\Gal(L'/K)$. After applying $b$ to the previous equality, a direct calculation gives
            \[
            b_{\rho_0} {}^{\rho}b_s {}^{\phi_{v,w}(s)} b_{\rho_0}^{-1} =b_{\psi} b_{\phi_{v,w}(s)}.
            \]
            Indeed, we have 
            \begin{align*}
                b_{\rho_0 s \rho_0^{-1}} &= b_{\rho_0} {}^{\rho_0}b_s{}^{\rho_0 s} b_{\rho_0^{-1}} \\ 
                &= b_{\rho_0} {}^{\rho}b_s {}^{\phi_{v,w}(s) \rho_0}b_{\rho_0^{-1}} \\ 
                & = b_{\rho_0} {}^{\rho}b_s {}^{\phi_{v,w}(s)} b_{\rho_0}^{-1}
            \end{align*}
            because $\phi_{v,w}(s)$ and $\rho_0 s \rho_0^{-1}$ have the same action on $F$. This implies that $b_{\psi} = 1$ and then $\psi = 1$ since $b$ induces an isomorphism between $\Gal(L'/L)$ and $F$. This proves the result.
        \end{proof}

    \subsection{Very weak approximation}
    We use the following definition of very weak approximation which is analogous to the classical case of number fields.
    \begin{defi} \label{def weak approximation}
        Let $X$ be a variety over $K$ and $T$ a subset of $C^{(1)}$. We say that $X$ has approximation in $T$ if the image of the diagonal morphism
        \[
            X(K) \to \prod_{v \in T} X(K_v)
        \]
        is dense with respect to the product topology. We also say that $X$ has approximation away from $T$ if $X$ has approximation in every finite subset $S$ of $C^{(1)}$ disjoint grom $T$. We say that $X$ has very weak approximation if there exists a finite set $T$ such that $X$ has approximation away from $T$.
    \end{defi}
    
    Let $G$ be a semi-simple simply connected linear algebraic group and $F$ a finite $K$-subgroup. The property of weak approximation for the variety $X = G/F$ only depends on $F$ as a $K$-group and it can be rephrased just in terms of Galois cohomology.
    \begin{defi} \label{def approximation outside S}
        Let $T$ be a subset of $C^{(1)}$. We say that a finite $K$-group $F$ has approximation in $T$ if the diagonal morphism
        \[
        \H^1(K,F) \to \prod_{v \in T}\H^1(K_v, F)
        \]
        is surjective. We also say that $F$ has approximation away from $T$ if it has approximation in every finite subset $S$ of $C^{(1)}$ disjoint from $T$. Moreover, $F$ is said to have very weak approximation if there exists a finite subset $T$ of $C^{(1)}$ such that $F$ has approximation away from $T$.
    \end{defi}
        Note that fields of type (a') and (b') satisfies Serre's conjecture II. Indeed, fields of type (a') are $C_2$ and for fields of type (b') see \cite{CTGP2004ArithLinGrpsTwoDim}. Then the argument in the introduction of  \cite{Luco2014ReductionFiniteStabilizer} can be adapted to prove the following proposition. 
        \begin{prop} \label{prop weak approx for homosp depends only on stab}
            Let $G$ be a semi-simple simply connected linear algebraic group and $F$ a finite $K$-subgroup. Denote by $X$ the variety $G/F$. Then for every finite subset $S$ of $C^{(1)}$, the variety $X$ has approximation in $S$ in the sense of Definition \ref{def weak approximation} if and only if $F$ has approximation in $S$ in the sense of Definition \ref{def approximation outside S}.
        \end{prop}
    The main result of this section is the following theorem.
    \begin{thm}\label{thm weak approx split extension}
        Let $K$ be either the function field of a curve over $\C(\!(t)\!)$ or a finite extension of $\C(\!(x,y)\!)$. Let $A$ be a finite abelian $K$-group and $F$ a finite $K$-group having approximation away from $\Bad_{F}$. Suppose that we have a split short exact sequence 
        \begin{equation*}
             \begin{tikzcd}
            0 \ar[r] & A \ar[r]& E \ar[r] & F \ar[r] \ar[l, bend right = 50]& 0.
            \end{tikzcd}
        \end{equation*}
        Then $E$ has approximation outside of $\Bad_E$.
    \end{thm}
    To prove this theorem, we will need the following lemma:
    \begin{lemma} \label{lemma weak approximation and sha}
        Let $A$ be a finite abelian $K$-group and $S$ a finite subset of $C^{(1)}$. Denote by $\widehat{A}$ the Cartier dual group i.e $\widehat{A} = \Hom(A,\mu )$. Then $A$ has approximation in $S$ if and only if 
        \[
            \Sha^1_S(K,\widehat{A}) = \Sha^1(K,\widehat{A}).
        \]
    \end{lemma}
    \begin{proof}
        When $K$ is a function field, this is non other than \cite[Lemme1.2]{diego2017LocGlobII}. When $K$ is a finite extension of $\C(\!(x,y)\!)$ one can carry out a very similar proof based on \cite[Theorem 3.5 and Proposition 3.7]{diego2019dual2dim}.
    \end{proof}

    \begin{proof}[Proof of Theorem \ref{thm weak approx split extension}]
        We follow the same steps as in \cite{Lu-Nef-Dem}. \\
        
        Let $S$ be a finite subset of $C^{(1)}$ disjoint from $\Bad_{E}$ and for every $v \in S$ fix a class $\beta_v \in \H^1(K_v,E)$. Let $L/K$ be the minimal extension splitting $E$. \\
        
        \textbf{Step 1:} construct a class $\beta' \in \H^1(K,E)$ having prescribed images in $\H^1(K_v,F)$ for $v \in S\cup S'$ where $S'$ is a suitable finite set of places. \\ 

        Consider the following commutative diagram, whose rows are exact in the sens of pointed sets
        \begin{equation}\label{diagrama aproximacion muy debil}
            \begin{tikzcd}
                \H^1(K,A) \ar[r] \ar[d] & \H^1(K,E) \ar[r] \ar[d] & \H^1(K,F) \ar[d] \ar[l,bend right = 30]\\ 
                \prod_{v \in C^{(1)}} \H^1(K_v,A) \ar[r] & \prod_{v \in C^{(1)}} \H^1(K_v,E) \ar[r] & \prod_{v \in C^{(1)}} \H^1(K_v,F) \ar[l,bend right = 30]
            \end{tikzcd}
        \end{equation}
        Denote by $\gamma_v$ the image of $\beta_v$ in $\H^1(K_v,F)$. Applying Lemma \ref{lemma reocurence cohomology} we can find for each $v \in S$ a place $v' \not \in S \cup \Bad_{E}$ together with isomorphisms 
        \[
                \rho_{v,v'}^*: \H^1(K_v,F) \to \H^1(K_{v'},F).
        \]
        Since there are infinitely many such $v'$ for each $v$, we may choose all the $v'$ pairwise different. Denote by $S'$ the set of the $v'$ we just chose. Let $\gamma_{v'}$ be $\phi^*_{v,v'}(\gamma_v)$. \par 
        Since we chose the places of $S'$ outside $\Bad_E$ and $F$ has approximation away from $\Bad_F \subseteq \Bad_E$, there exists a class $\gamma \in \H^1(K,F)$ that restricts to $\gamma_v$ for every $v \in S\cup S'$. Moreover, we can assume that a cocycle $c$ representing $\gamma$ restricted to $\Gamma_L$ is surjective. Indeed, by Proposition \ref{prop existence totally split} and Theorem \ref{thm reocurrence} there is an infinite number of places $w \not \in S \cup S' \cup \Bad_E$ that are totally split on the extension $L/K$. In particular, for every such $w$, the absolute Galois group of $K_w$ acts trivially on $F$. For every conjugacy class of $F$ we can choose a different place $w \not \in S \cup S' \cup \Bad_E$ and an unramified class $\gamma_w \in \H^1(K_w,F)$ represented by a morphism sending a topological generator of $\Gamma_{K_w}^{nr}$ to an element in the conjugacy class. Adding these local conditions to $S \cup S'$ forces surjectivity of $c$. \par 

         Denote by $\beta'$ the image of $\gamma$ by the section $\H^1(K,F) \to \H^1(K,E)$. \\

        \textbf{Step 2:} reducing to proving $\Sha_S^1(K,\widehat{{}_cA}) = \Sha^1(K,\widehat{{}_cA})$. \\
        
        We can twist the diagram \eqref{diagrama aproximacion muy debil} by $c$, a class representing $\gamma$, to get the following commutative diagram with exact rows
         \begin{equation*}
            \begin{tikzcd}
                \H^1(K,{}_cA) \ar[r] \ar[d] & \H^1(K,{}_cE) \ar[r] \ar[d] & \H^1(K,{}_cF) \ar[d] \ar[l,bend right = 30]\\ 
                \prod_{v \in \Omega} \H^1(K_v,{}_cA) \ar[r] & \prod_{v \in \Omega} \H^1(K_v,{}_cE) \ar[r] & \prod_{v \in \Omega} \H^1(K_v,{}_cF). \ar[l,bend right = 30]
            \end{tikzcd}
            \end{equation*}
            By construction we know that the images of $\beta'$ and $\beta_v$ in $\H^1(K_v,F)$ are $\gamma_v$ for $v \in S\cup S'$ and, similarly, the images of ${}_c\beta'$ and ${}_c \beta_v$ in $\H^1(K_v,{}_cF)$ are ${}_c\gamma_v$ for every $v \in S \cup S'$. Since the class ${}_c \gamma_v$ is the distinguished element of the pointed set $\H^1(K_v,{}_cF)$, the class ${}_c\beta_v$ comes from $\alpha_v \in \H^1(K_v, {}_cF)$. In order to conclude, we need to find a class $\alpha \in \H^1(K,{}_cA)$ such that its restriction is $\alpha_v$ for every $v \in S$. Hence it suffices to prove $\Sha_S^1(K,\widehat{{}_cA}) = \Sha^1(K,\widehat{{}_cA})$ thanks to Lemma \ref{lemma weak approximation and sha}.\\
           
        \textbf{Step 3:} proving $\Sha_S^1(K,\widehat{{}_cA}) = \Sha^1(K,\widehat{{}_cA})$ \\

        Let $L'$ be the extension corresponding to the kernel of $c: \Gamma_L \to F$ and take $\alpha \in \Sha^1_S(K,\widehat{{}_cA})$. Fix $v \in S$ and $v' \in S'$ corresponding to $v$ as in the first step. Applying Lemma \ref{lemma cocylces fun calculation} we see that the decomposition groups of $v$ and $v'$ are conjugates in $\Gal (L'/K)$. Note that $L'$ splits $\widehat{{}_cA}$ because $K$ contains all the root of unity and $L'$ splits ${}_cA$. \par 
            We can apply Lemma \ref{prop Sha in the image of inf} to see that $\alpha$ is in the image of the inflation map $\Inf: \H^1(\widetilde{L'}/K,\widehat{{}_cA}) \to \H^1(K, \widehat{{}_cA})$. The fact that the decomposition groups of $v$ and $v'$ are conjugates implies the existence of an isomorphism 
            \[
            \H^1(L'_v/K_v,\widehat{{}_cA}) \xrightarrow{\simeq} \H^1(L'_{v'}/K_{v'},\widehat{{}_cA})
            \]
            compatible with restrictions, meaning that the following triangle commutes
            \begin{equation*}
                \begin{tikzcd}[column sep =1ex, row sep = 7ex]
                    &\H^1(\widetilde{L'}/K,\widehat{{}_cA}) \ar[rd] \ar[dl]& \\    
                \H^1(L'_{v}/K_v, \widehat{{}_c A}) \ar[rr] & &  \H^1(L'_{v'}/K_{v'},\widehat{{}_c A}).                
                \end{tikzcd}
            \end{equation*}
            This commutativity implies that $\alpha_v = 0$ because we assume that $v' \not \in S$ and $\alpha \in \Sha_S^1(K,\widehat{{}_cA})$. Therefore $\alpha$ is in $\Sha^1(K,\widehat{{}_cA})$.
    \end{proof}    
    \appendix 
        \section{Reocurence for henselian fields}
        Let $k_0$ be a field of characteristic zero. We denote by $k_0[x,y]^h$ the henselisation of the local ring $k_0[x,y]_{(x,y)}$ and by $k_0(x,y)^h$ its fraction field, see \cite[\S 2.3]{NeronModels} for details. Similarly for one variable, that is $k_0[x]^h$ is the henselisation of the local ring $k_0[x]_{(x)}$ and $k_0(x)^h$ its fraction field. For a finite extension $K$ of $k_0(x,y)^h$ we will denote by $C$ the Dedekind scheme $\Spec A[x^{-1}]$ where $A$ is the integral closure of $k[x,y]^h$ in $K$. \par 
        In this appendix we give a proof of Theorem \ref{thm reocurrence} and Theorem \ref{thm finer reocurrence} in the case of a finite field extension of $k_0(x,y)^h$. The methods used to prove these results for finite field extensions of $k_0(\!(x,y)\! )$ can be adapted to give a proof for finite extensions of $k_0(x,y)^h$, but in this appendix we give a different proof based on the idea of ``approximating'' the field $k_0(x,y)^h$ by $k_0(x)^h(y)$.
    \begin{lemma}\label{lemma density for every valuation}
        Let $v$ be in $\Spec k_0[x,y]^h[x^{-1}]$. Then the subfield $k_0 (x )^h (y)$ is dense in $k_0(x,y)^h$ with respect to the $v$-adic topology.
    \end{lemma}
    \begin{proof}
    Let $\pi_v \in k[x,y]^h$ be a uniformizer at $v$. Since $v$ is not $(x)$ we can apply \cite[Theorem 1.1]{HenselianPerp} to $\pi_v$ inside the ring $k_0[x]^h[y]^h$. Hence we may assume that $\pi_v$ is in $k_0[x]^h[y]$. For every $f \in k_0 [x,y]^h$ we can apply \cite[Theorem 1.1]{HenselianPerp} to get a unique decomposition
    \[
        f = q_0 \pi_v + r_0 
    \]
    with $q_0 \in k_0[x]^h[y]^h$ and $r_0 \in k_0[x]^h[y]$. Using the same argument for $q_0$ and continuing inductively, we get a decomposition
    \[
        f = q_n\pi_v^{n+1} + \cdots + r_1\pi_v + r_0
    \]
    with $r_i \in k_0[x]^h[y]$. Define $f_n$ to be $r_{n-1} \pi_v^{n-1} + \cdots + r_0 \in k_0[x]^h[y]$. We clearly have
    \[
        v(f-f_n) \geq n+1.
    \]
    Then every element of $k_0[x,y]^h$ can be arbitrarily approximated by one of $k_0[x]^h[y]$ with respect to the valuation $v$. In other words, the field $k_0(x)^h(y)$ is dense with respect to the $v$-adic topology. 
    \end{proof}
    Note that in the previous proof we can replace $k_0(x,y)^h$ by $k_0(\!(x,y)\!)$ and $k_0(x)^h(y)$ by $k_0(\!(x)\!)(y)$ to get an analogous result. Moreover, we can use the usual Weierstarss preparation theorem for formal power series instead of \cite[Theorem 1.1]{HenselianPerp}.
    \begin{thm} \label{thm reocurrence henselian}
        Let $K$ be a finite extension of $k_0(x,y)^h$ and $L/K$ a finite Galois extension with group $G$. Let $v \in C^{(1)}$ unramified in $L/K$. Then, there exists infinitely many places $w$ of $K$ such that the decomposition groups $G_{v}$ and $G_w$ are conjugates.
    \end{thm}
    \begin{proof}
        We begin by constructing fields $L_0$, $L_0^{\mathrm{gal}}$ and $K_0$ fitting in the following diagram
        
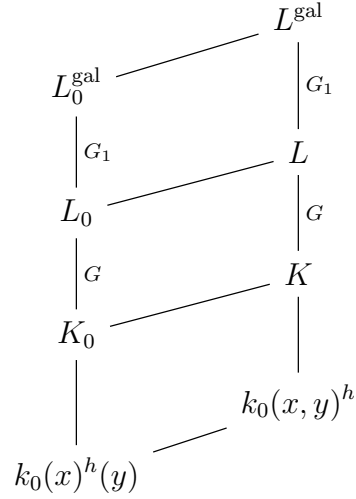
\begin{equation*}
            \begin{tikzcd}[row sep = 1ex]
                    & L^{\mathrm{gal}}\\
                    L_0^{\mathrm{gal}} \ar[ur,dash] & \\
                    & L \ar[uu,dash,swap, "G_1"]\\
                    L_0 \ar[uu,dash,swap,"G_1"] \ar[ur,dash] & \\
                    & K\ar[uu,dash,swap,"G"] \\
                    K_0\ar[uu,dash,swap,"G"] \ar[ur,dash] & \\
                    & k_0(x,y)^h\ar[uu,dash] \\
                    k_0(x)^h(y) \ar[uu,dash] \ar[ur,dash]
            \end{tikzcd}
        \end{equation*}
        where $L^{\mathrm{gal}}$ is the Galois closure of $L/ k_0(x,y)^h$. \par 
        Since the extension $L^{\mathrm{gal}}/k_0(x,y)^h$ is separable, there exists a primitive element $\alpha \in L^{\mathrm{gal}}$. Define $L_0^{\mathrm{gal}}$ to be $k_0(x,y)^h(\alpha)$. Note that we can identify the Galois group of $L^{\mathrm{gal}}/k_0(x)^h(y)$ as a subgroup of $\Gal(L_0^{\mathrm{gal}}/k_0(x)^h(y))$. In particular, $G_1 = \Gal(L^{\mathrm{gal}}/L)$ is a subgroup of $\Gal(L^{\mathrm{gal}}_0 / k_0(x)^h(y))$. Denote by $L_0$ the subfield of $L_0^{\mathrm{gal}}$ fixed by $G_1$. Finally, using that $L_0$ is a subfield of $L$ we define $K_0$ to be the subfield of $L_0$ fixed by $G$. Note that the extension $L_0/K_0$ is Galois with group $G$. \par
        
        We need to compare the decomposition groups of $L_0/K_0$ and $L/K$. Let $D$ be the Dedekind scheme corresponding to $L$ and $w$ be a place in $D^{(1)}$ above $v$. Denote by $v_0$ and $w_0$ the restriction of $v$ to $K_0$ and $w$ to $L_0$ respectively. We claim that the decomposition groups $G_{v_0}$ and $G_v$ defined as stabilisers of $w_0$ and $w$ respectively are equal. Indeed, the inclusion $G_{v} \subseteq G_{v_0}$ is straightforward because $w_0$ is the restriction of $w$. For the converse, take $\sigma \in G_{v_0}$ and $f \in L$. Using Lemma \ref{lemma density for every valuation} and continuity of $\sigma$ with respect to $w$, we can find $f_0 \in L_0$ such that $w(f) = w(f_0)$ and $w(\sigma^{-1}(f)) = w(\sigma^{-1}(f_0))$. Hence $\sigma$ is in $G_v$ because $w(\sigma^{-1}(f)) = w(\sigma^{-1}(f_0)) = w(f_0) = w(f)$. \par 
        Note that $v_0$ is unramified in $L_0$ because $v$ is unramified in $L/K$. Denote by $C_0$ the Dedekind scheme corresponding to $K_0$. Applying Theorem \ref{thm reocurrence} to $K_0$, we get infinitely many places $u_0\in C_0^{(1)}$ such that $G_{u_0}$ and $G_{v_0}$ are conjugates. Every $u_0$ has finitely many valuations $u \in C^{(1)}$ above because the morphism $C \to C_0$ has finite fibers. We conclude using that the decomposition groups coincide.
    \end{proof}
    \bibliography{ref}
    \bibliographystyle{alpha}
\end{document}